\documentclass[12pt]{amsart}
\usepackage{amssymb,amsfonts,amsmath,graphicx,lineno}
\usepackage{mathrsfs}
\usepackage{anysize}
\usepackage[colorlinks, citecolor=black, urlcolor=black, linkcolor=black]{hyperref}
\usepackage{url}
\usepackage{color}
\usepackage{tikz}
\usetikzlibrary{arrows}
\usepackage{pgfplots}
\tikzset{>= angle 60}
\theoremstyle{plain}
\newtheorem{theorem}{Theorem}[section]
\newtheorem{lemma}[theorem]{Lemma}
\newtheorem{corollary}[theorem]{Corollary}
\newtheorem{proposition}[theorem]{Proposition}
\theoremstyle{definition}
\newtheorem{definition}[theorem]{Definition}
\theoremstyle{remark}

\newcommand{\R}{{\mathbb{R}}}

\newcommand{\cqd}{\hfill $\square$} 

\title[Piecewise $\lambda$-affine maps ]{Topological dynamics of piecewise $\lambda$-affine maps}


\medskip
\begin{document}

\maketitle

%% Enter the first author's name and address:
\centerline{\scshape Arnaldo Nogueira\footnote{Partially
  supported by BREUDS.}}
\smallskip
{\footnotesize
 %% please put the address of the first author
 \centerline{Aix-Marseille Universit\'e, CNRS, Centrale Marseille, I2M, UMR 7373, 13453 Marseille, France}
 \centerline{arnaldo.nogueira@univ-amu.fr}} %% Do not forget to end the {\footnotesize by the sign }

\medskip

\centerline{\scshape Benito Pires \footnote{Partially supported by FAPESP 2015/20731-5 and CNPq 303731/2013-3.} and Rafael A. Rosales}
\smallskip

{\footnotesize
 %% please put the address of the second author
 \centerline{Departamento de Computa\c c\~ao e Matem\'atica, Faculdade de Filosofia, Ci\^encias e Letras}
 \centerline {Universidade de S\~ao Paulo, 14040-901, Ribeir\~ao Preto - SP, Brazil}
   \centerline{benito@usp.br, rrosales@usp.br} }

\marginsize{2.5cm}{2.5cm}{1cm}{2cm}
  \begin{abstract} Let $-1<\lambda<1$ and $f:[0,1)\to\R$ be a piecewise $\lambda$-affine map, that is, there exist points $0=c_0<c_1<\cdots <c_{n-1}<c_n=1$ and real numbers
  $b_1,\ldots,b_n$ such that $f(x)=\lambda x+b_i$ for every $x\in [c_{i-1},c_i)$. We prove that, for Lebesgue almost every $\delta\in\R$, the map $f_{\delta}=f+\delta\,({\rm mod}\,1)$
  is asymptotically periodic. More precisely, $f_{\delta}$ has at most $2n$ periodic orbits and the $\omega$-limit set of every $x\in [0,1)$ is a periodic orbit.         \end{abstract}

\section{Introduction}

 Let $I=[0,1)$ and $-1<\lambda<1$. We say that $f:I\to \R$ is an  {\it $n$-interval piecewise $\lambda$-affine map} if there exist points $0=c_0<c_1<\cdots <c_{n-1}<c_n=1$  and  real numbers 
 $b_1,\ldots, b_n$  
such that  
$f(x)=\lambda x + b_i$ for every $x\in [c_{i-1},c_i)$ and $1\le i\le n$. We are interested in studying the topological dynamics of the one-parameter family of piecewise $\lambda$-affine contractions (see Figure \ref{figthm1})
\begin{equation}\label{fdelta}
 f_{\delta}=f + \delta\,({\rm mod}\,1), \quad \delta\in\R.
\end{equation}

The case in which $0<\lambda<1$ and $f$ is the continuous map $x\mapsto \lambda x$ was explicitly considered   by, among others, Y. Bugeaud  \cite{YB}, Y. Bugeaud  and J-P. Conze \cite{YBC}, R. Coutinho \cite{RC} and P. Veerman \cite{PV}  using a rotation number approach. It is known that for each $\delta\in\R$, the $\omega$-limit set $\omega_{f_{\delta}}(x)=\bigcap_{m\ge 0}\overline{\bigcup_{k\ge m}\{ f_\delta^k(x)\}}$ is the same set for every $x\in I$: either a finite set or a Cantor set. The second situation happens for a non-trivial Lebesgue null set of parameters $\delta$.

Here we consider the general case where $f$ is any piecewise $\lambda$-affine contraction having finitely many discontinuities. Beyond the difficulty brought by the presence of discontinuities, we also have  to deal with the possible lack of injectivity of the map, which rules out any approach based on the  theory of rotation numbers. Differently from the $\delta$-parameter family
$x\mapsto \lambda x +\delta\,({\rm mod}\,1)$, the dynamics of the general case allows the coexistence of several attractors of finite cardinality together with several Cantor sets. In other words, $\omega_{f_{\delta}}(x)$ may depend on $x$. 

Given $f:I\to I$ and $x\in I$, if there exists $k\geq1$ such that $f^k(x)=x$, we say that  the {\it $f$-orbit} of $x$, $O_f(x)=\bigcup_{k\ge 0}\{ f^k(x)\}$,  is a {\it periodic orbit}.
We say that $f$ is {\it asymptotically periodic} if $\omega_f(x)$ is a periodic orbit for every $x\in I$. 

Our first result is the following. 

\begin{theorem}\label{abcd} Let $-1<\lambda<1$ and $f:I\to\R$ be an $n$-interval piecewise $\lambda$-affine map, then, for Lebesgue almost every $\delta\in\R$, the map $f_{\delta}=f+\delta\,({\rm mod}\,1)$ is asymptotically periodic and has at most $2n$ periodic orbits.
\end{theorem}
In the statement of Theorem \ref{abcd}, the bound $2n$ for the number of periodic orbits is sharp: in fact, for $n=1$ and $f:x\mapsto -\frac{x}{2}+\frac{1}{4}$ we have that $f_{\delta}$ is the map
$x\mapsto -\frac{x}{2}+\frac{1}{4} +\delta\,({\rm mod}\,1)$, which has two periodic orbits for every $\delta$ small enough.  However,
the bound $2n$ can be replaced by $n$, if in $(\ref{fdelta})$ the map $f$ satisfies $f(I)\subset (0,1)$. Besides, the claim of Theorem \ref{abcd} holds for $I=\R$, $f_{\delta}=f+\delta$ and the bound $n$ in the place of $2n$. Observe that being asymptotically periodic is stronger than saying that $\omega_{f_{\delta}}(x)$ is a finite set for every $x\in I$. Due to the arguments in the proof of Lemma \ref{redl}, our approach to prove Theorem \ref{abcd} requires $f$ to be a constant slope map .

\begin{figure}[h]
 \begin{center}
  %\documentclass[border=10pt]{standalone}
%\usepackage{amsmath}
%\usepackage{tikz}
%\usetikzlibrary{arrows}
%\usepackage{pgfplots}
%\tikzset{>= angle 60}

%\begin{document}
% border for outher continuous box
\newcommand*{\slb}{0.03}%
% delta neighborhood: 
\newcommand*{\dlt}{0.111111}%
\newcommand*{\yone}{5/18}
\newcommand*{\ytwo}{13/18}
% line width
\newcommand*{\lwd}{1.2pt}
% define a value for delta,
\newcommand*{\dONE}{2/9}
\newcommand*{\dTWO}{14/25}

\begin{tikzpicture}[scale=3.95]
 %~~~~~~~~~~~~~~~~~~~~~~~~~~~~~~~~~~~~~~~~~~~~~~~~~~~~~~~~~~~
 %      Plot 1: f(x) = 1/2*x + b_i, for n=4 and
 %                     b_1 = 1/3
 %                     b_2 = 2/3
 %                     b_3 = 2/5
 %                     b_4 = 1/10
 %~~~~~~~~~~~~~~~~~~~~~~~~~~~~~~~~~~~~~~~~~~~~~~~~~~~~~~~~~~~
\begin{scope}[shift={(0,0)}]
  \draw 
     (0-\slb,0-\slb) -- (1+\slb,0-\slb) -- (1+\slb,1+\slb) 
       -- (0-\slb,1+\slb) -- cycle;
  \draw[line width=0.25pt, densely dashed] 
      (0,0) -- (1,0) -- (1,1) -- (0,1) -- cycle;
   % constant slope PC function comes here
   \draw[line width=\lwd] (0,1/3) --   (0.3, 0.48333);
   \draw[line width=\lwd] (0.3,2/3) -- (0.6,1/2*3/10+2/3);
   \draw[line width=\lwd] (0.6,2/5) -- (0.9,1/2*3/10+2/5);
   \draw[line width=\lwd] (0.9,1/10) -- (1,1/2*1/10+1/10);
   \foreach \x in {0.3, 0.6, 0.9} {
     \draw[line width=0.4pt, dotted]  (\x,0) -- (\x,1);}
   % endpoints
   \draw[fill=black] (0,1/3) circle (0.4pt);
   \draw[fill=white] (0.3, 0.48333) circle (0.4pt);
   \draw[fill=black] (0.3,2/3) circle (0.4pt);
   \draw[fill=white] (0.6,1/2*0.3+2/3) circle (0.4pt);
   \draw[fill=black] (0.6,2/5) circle (0.4pt);
   \draw[fill=white] (0.9,1/2*0.3+2/5) circle (0.4pt);
   \draw[fill=black] (0.9,0.1) circle (0.4pt);
   \draw[fill=white] (1,1/2*0.1+0.1) circle (0.4pt);
   % x and y tics
   \foreach \x in {0, 0.3, 0.6, 0.9, 1} {
      \draw (\x, -\slb) -- (\x, -\slb-0.02);}
   \foreach \x in {0, 1} {
      \draw (-\slb, \x) -- (-\slb-0.02, \x);}
   % x = y
  % \draw[line width=0.4pt, dotted] (0,0) -- (1,1);
  % labels
  \node at (0, -0.1) {0};
  \node at (1, -0.1) {1};
  \node at (-0.1, 0) {0};
  \node at (-0.1, 1) {1};
  \node at (0.3, -0.1) {\small $c_1$};
  \node at (0.6, -0.1) {\small $c_2$};
  \node at (0.9, -0.1) {\small $c_3$};
  % title
  \node at (1/2, -0.3) {\small $f$ with $\lambda=\frac{1}{2}$, $n=4$};
 \end{scope}
 %~~~~~~~~~~~~~~~~~~~~~~~~~~~~~~~~~~~~~~~~~~~~~~~~~~~~~~~~~~~
 %      Plot 2: f + \dONE (mod 1)
 %~~~~~~~~~~~~~~~~~~~~~~~~~~~~~~~~~~~~~~~~~~~~~~~~~~~~~~~~~~~
\begin{scope}[shift={(1.4,0)}]
  \draw 
     (0-\slb,0-\slb) -- (1+\slb,0-\slb) -- (1+\slb,1+\slb) 
       -- (0-\slb,1+\slb) -- cycle;
  \draw[line width=0.25pt, densely dashed] 
      (0,0) -- (1,0) -- (1,1) -- (0,1) -- cycle;
   % f + \dONE
   \draw[line width=\lwd] (0,1/3+\dONE) --   (0.3, 0.48333+\dONE);
   \draw[line width=\lwd] (0.3,2/3+\dONE) -- (0.5222,1); % 0.5222=(1-(2/3+\dONE)+0.3/2)*2
   \draw[line width=\lwd] (0.5222,0) -- (0.6,0.0389); % 0.0389 = 1/2*(0.6-0.5222)
   \draw[line width=\lwd] (0.6,2/5+\dONE) -- (0.9,1/2*0.3+2/5+\dONE);
   \draw[line width=\lwd] (0.9,0.1+\dONE) -- (1,1/2*0.1+0.1+\dONE);
   %% old f values
   %\draw[line width=\lwd, gray!20] (0,1/3) --   (0.3, 0.48333);
   %\draw[line width=\lwd, gray!20] (0.3,2/3) -- (0.6,1/2*0.3+2/3);
   %\draw[line width=\lwd, gray!20] (0.6,2/5) -- (0.9,1/2*0.3+2/5);
   %\draw[line width=\lwd, gray!20] (0.9,0.1) -- (1,1/2*0.1+0.1);
   %% old f endpoints
   %\draw[fill=gray!20, gray!20] (0,1/3) circle (0.4pt);
   %\draw[gray!20, fill=white] (0.3, 0.48333) circle (0.4pt);
   %\draw[fill=gray!20, gray!20] (0.3,2/3) circle (0.4pt);
   %\draw[gray!20, fill=white] (0.6,1/2*0.3+2/3) circle (0.4pt);
   %\draw[fill=gray!20, gray!20] (0.6,2/5) circle (0.4pt);
   %\draw[gray!20, fill=white] (0.9,1/2*0.3+2/5) circle (0.4pt);
   %\draw[fill=gray!20, gray!20] (0.9,0.1) circle (0.4pt);
   %\draw[gray!20, fill=white] (1,1/2*0.1+0.1) circle (0.4pt);  
   % vertical lines
   \foreach \x in {0.3, 0.5222222, 0.6, 0.9} {
      \draw[line width=0.4pt, dotted]  (\x,0) -- (\x,1);}
   % endpoints
   \draw[fill=black] (0,1/3+\dONE) circle (0.4pt);
   \draw[fill=white] (0.3, 0.48333+\dONE) circle (0.4pt);
   \draw[fill=black] (0.3,2/3+\dONE) circle (0.4pt);
   \draw[fill=white] (0.5222,1) circle (0.4pt);
   \draw[fill=black] (0.5222,0) circle (0.4pt);
   \draw[fill=white] (0.6,0.0389) circle (0.4pt);
   \draw[fill=black] (0.6,2/5+\dONE) circle (0.4pt);
   \draw[fill=white] (0.9,1/2*0.3+2/5+\dONE) circle (0.4pt);
   \draw[fill=black] (0.9,0.1+\dONE) circle (0.4pt);
   \draw[fill=white] (1,1/2*0.1+0.1+\dONE) circle (0.4pt);      
   % x and y tics
   \foreach \x in {0, 0.3, 0.5222222, 0.6, 0.9, 1} {
      \draw (\x, -\slb) -- (\x, -\slb-0.02);}
   \foreach \x in {0, 1} {
      \draw (-\slb, \x) -- (-\slb-0.02, \x);}
  % x = y
 % \draw[line width=0.4pt, dotted] (0,0) -- (1,1);
  % labels
  \node at  (0, -0.1) {0};
  \node at (1, -0.1) {1};
  \node at  (-0.1, 0) {0};
  \node at (-0.1, 1) {1};
  \node at (0.3, -0.1) {\small $c_1$};
  \node at (0.5222222, -0.1) {\small $c_2$};
  \node at (0.62, -0.1) {\small $c_3$};
  \node at (0.9, -0.1) {\small $c_4$};  
  % title
  \node at (1/2, -0.3) {\small $f _\delta$ with $\delta= \frac{2}{5}$};
 \end{scope} 
 %~~~~~~~~~~~~~~~~~~~~~~~~~~~~~~~~~~~~~~~~~~~~~~~~~~~~~~~~~~~
 %      Plot 3: f + \dTWO (mod 1)
 %~~~~~~~~~~~~~~~~~~~~~~~~~~~~~~~~~~~~~~~~~~~~~~~~~~~~~~~~~~~
\begin{scope}[shift={(2.8,0)}]
  \draw 
     (0-\slb,0-\slb) -- (1+\slb,0-\slb) -- (1+\slb,1+\slb) 
       -- (0-\slb,1+\slb) -- cycle;
  \draw[line width=0.25pt, densely dashed] 
      (0,0) -- (1,0) -- (1,1) -- (0,1) -- cycle;
   % constant slope PC function comes here
   % constant slope PC function comes here
   \draw[line width=\lwd] (0,1/3+\dTWO) --   (16/75, 1);
   \draw[line width=\lwd] (16/75,0) --   (0.3, 13/300); % new piece
   \draw[line width=\lwd] (0.3,17/75) -- (0.6,113/300);
   \draw[line width=\lwd] (0.6,2/5+\dTWO) -- (17/25,1);
   \draw[line width=\lwd] (17/25,0) --   (0.9,11/100); % new piece   
   \draw[line width=\lwd] (0.9,0.1+\dTWO) -- (1,1/2*0.1+0.1+\dTWO);
   \foreach \x in {16/75, 0.3, 0.6, 17/25, 0.9} {
     \draw[line width=0.4pt, dotted]  (\x,0) -- (\x,1);}
   % endpoints
   \draw[fill=black] (0,1/3+\dTWO) circle (0.4pt);
   \draw[fill=white] (16/75, 1) circle (0.4pt);
   \draw[fill=black] (16/75,0) circle (0.4pt);
   \draw[fill=white] (0.3, 13/300) circle (0.4pt);
   \draw[fill=black] (0.3,17/75) circle (0.4pt);
   \draw[fill=white] (0.6,113/300) circle (0.4pt);
   \draw[fill=black] (0.6,2/5+\dTWO) circle (0.4pt);
   \draw[fill=white] (17/25,1) circle (0.4pt);
   \draw[fill=black] (17/25,0) circle (0.4pt);
   \draw[fill=white] (0.9,11/100) circle (0.4pt);      
   \draw[fill=black] (0.9,0.1+\dTWO) circle (0.4pt);
   \draw[fill=white] (1,1/2*0.1+0.1+\dTWO) circle (0.4pt);      
   % x and y tics
   \foreach \x in {0, 16/75, 0.3, 0.6, 17/25, 0.9, 1} {
      \draw (\x, -\slb) -- (\x, -\slb-0.02);}
   \foreach \x in {0, 1} {
      \draw (-\slb, \x) -- (-\slb-0.02, \x);}
  % x = y
  %\draw[line width=0.4pt, dotted] (0,0) -- (1,1);
  % labels
  \node at  (0, -0.1) {0};
  \node at (1, -0.1) {1};
  \node at  (-0.1, 0) {0};
  \node at (-0.1, 1) {1};
  \node at  (16/75, -0.1) {\small $c_1$};
  \node at (0.3+0.02, -0.1) {\small $c_2$};
  \node at (0.6, -0.1) {\small $c_3$};  
  \node at  (17/25+0.02, -0.1) {\small $c_4$};
  \node at (0.9, -0.1) {\small $c_5$};
  
  % title
  \node at (1/2, -0.3) {\small $f_\delta$ with $\delta=\frac{14}{25}$};
 \end{scope}
\end{tikzpicture}
%\end{document}
  \caption{The family $f_{\delta}=f+\delta$ ({\rm mod}\,1).}\label{figthm1}
 \end{center}
\end{figure}
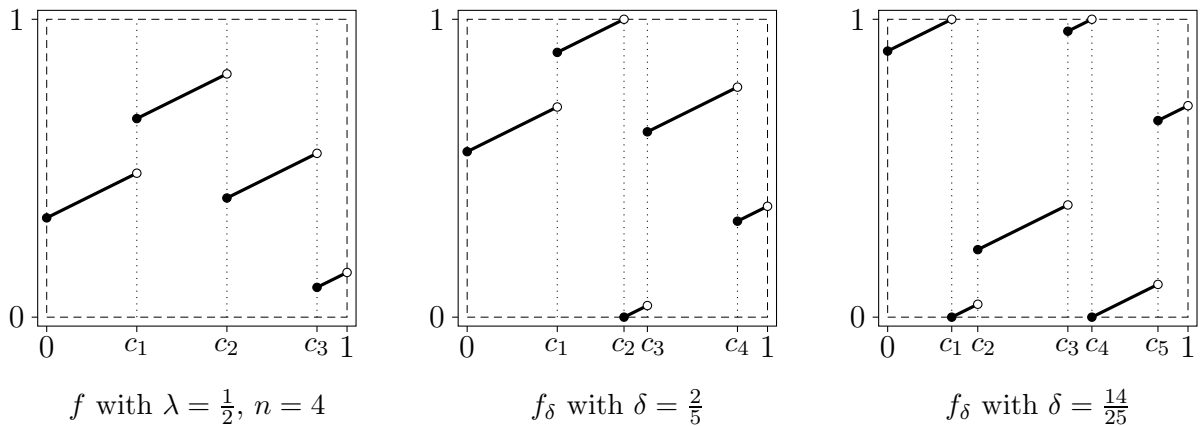

Maps of constant slope are important because many piecewise smooth interval maps are topologically conjugate or semiconjugate to them. In this regard, J. Milnor and W. Thurston \cite{MT}  proved  that any continuous piecewise monotone interval map of positive entropy ${\rm htop}\,(T)$ is topologically semiconjugate to a map whose slope in absolute value equals $e^{{\rm htop}}$. This result was generalised by L. Alsed\`a and M. Misiurewicz \cite{AM} to piecewise continuous piecewise monotone interval maps of positive entropy. Concerning countably piecewise continuous piecewise monotone interval maps, a necessary and sufficient condition for the existence of a non-decreasing semiconjugacy to a map of constant slope was provided by
  M. Misiurewicz and S. Roth  \cite{MR}.  A. Nogueira and B. Pires  \cite{NP} proved  that every injective piecewise contraction is topologically conjugate to a  map whose slope in absolute value equals $\frac12$. It is worth observing that the types of maps we consider here appear in the field of diophantine approximation (see \cite{YB2}). Concerning the dynamics of piecewise contractions, we refer the reader to \cite{BD,CGMU,GT}.

Theorem \ref{abcd} can be reduced to Theorem \ref{main}, a much more general result. To state it, we need some additional notation. Let $I=[0,1)$ or $I=\R$. Denote by $\overline{I}$ and $\mathring{I}$, respectively, the closure and the interior of $I$. We say that $\Phi=\{\phi_1,\ldots,\phi_n\}$, $n\ge 2$,  is an {\it Iterated Function System} (IFS)  defined on $I$ if each map $\phi_i:\overline{I}\to\mathring{I}$ is a Lipschitz contraction. Set $\inf \,\R=-\infty$, $\sup \,\R=\infty$ and 
$$
\Omega_{n-1}=\Omega_{n-1}(I)=\{(x_1,\ldots,x_{n-1}): \inf I<x_1<\cdots <x_{n-1}<\sup I\}.
$$ 
For each $(x_1,\ldots,x_{n-1})\in\Omega_{n-1}(I)$, let 
$f_{\phi_1,\ldots,\phi_n,x_1,\ldots,x_{n-1}}: I \to I$ be the {\it $n$-interval piecewise contraction $($PC$)$}  defined by
\begin{equation}\label{mapf}
 f_{\phi_1,\ldots,\phi_n,x_1,\ldots,x_{n-1}}(x)=
 \begin{cases}
 \phi_1(x)\quad\textrm{if}\quad x\in I\cap (-\infty,x_1)\\
 \phi_i(x)\quad\textrm{if}\quad x\in [x_{i-1},x_i),\quad 2\le i\le n-1\,.\\
\phi_n(x)\quad\textrm{if}\quad x\in I\cap [x_{n-1},\infty) \end{cases}
\end{equation} 
All measure-theoretical statements hereafter concern the Lebesgue measure denoted by $\mu$. In particular, $W_{\Phi}\subset I$ is a {\it full set} if $\mu(I\setminus W_{\Phi})=0$.
  
\begin{theorem}\label{main} Let $I=[0,1)$ or $I=\R$. Let $\Phi=\{\phi_1,\ldots,\phi_n\}$ be an IFS defined on $I$, then there exists a full set $W_{\Phi}\subset I$ such that
for  every $(x_1,\ldots,x_{n-1})\in\Omega_{n-1}(I)\cap W_{\Phi}^{n-1}$, the $n$-interval PC $f_{\phi_1,\ldots,\phi_n,x_1,\ldots,x_{n-1}}$ defined by $($\ref{mapf}$)$ is  asymptotically periodic and has at most $n$ periodic orbits.
\end{theorem}
     
 Notice that in Theorem \ref{main}, the IFS $\Phi$ does not need to be affine nor injective. A weaker version of Theorem \ref{main} was proved by the authors in \cite{NPR} under two additional hypothesis: the maps $\phi_1,\ldots,\phi_n$ were injective and had non-overlapping ranges.
 
 This article is organized as follows. The proof of Theorem \ref{main} for $I=[0,1)$ is distributed along Sections \ref{IFS}, \ref{epi}, \ref{gc} and \ref{upperb}. The first three sections are dedicated to the asymptotic stability aspect while the upper bound for the number of periodic orbits  is in Section \ref{upperb}. The proof of Theorem \ref{main} for $I=\R$ is in Section \ref{proofofmain}.  Theorem \ref{abcd} is proved in Section \ref{proofofabcd}.

 \section{Highly Contractive Iterated Function Systems}\label{IFS}

In this section we provide a version of Theorem \ref{main}  for highly contractive IFSs defined on $I=[0,1)$ as follows. If $\phi: I \rightarrow I$ is a Lipschitz map, then $D\phi(x)$ exists  for almost every $x\in I$.
We say that an IFS $\{\phi_1,\ldots,\phi_{n}\}$ is {\it highly contractive} if there exists $0\le\rho<1$ such that, for almost every $x\in I$,
\begin{equation}\label{rho}
\vert D\phi_1(x)\vert +\ldots +\vert D\phi_n(x)\vert \le \rho<1.
\end{equation}

\begin{theorem}\label{pr1} 
Let $\Phi=\{\phi_1,\ldots,\phi_{n}\}$ be a highly contractive IFS defined on $I=[0,1)$, then there exists a full set 
$W_{\Phi}\subset I$ such that, for every $(x_1,\ldots,x_{n-1})\in\Omega_{n-1}\cap W_{\Phi}^{n-1}$, the  PC 
$f_{\phi_1,\ldots,\phi_n,x_1,\ldots,x_{n-1}}$ defined by $($\ref{mapf}$)$  is asymptotically periodic.
\end{theorem}

We need some preparatory lemmas to prove Theorem \ref{pr1}. Throughout this section, except in Definition \ref{partition} and Lemma \ref{tmt}, 
we assume that $\Phi$ is a highly contractive IFS.

Denote by $Id$ the identity map on $\bar{I}$. Let $\mathscr{C}_0=\{Id\}$ and $A_0=\bar{I}$. For every $k\ge 0$, let
\begin{equation}\label{ckak}
\mathscr{C}_{k+1}=\mathscr{C}_{k+1}(\phi_1,\ldots,\phi_n)=\{\phi_{i}\circ h : 1\le i\le n, h \in \mathscr{C}_{k} \} \;\; \mbox{ and } \; A_k=\cup_{h\in\mathscr{C}_k} h(\bar{I}).
\end{equation}

\begin{lemma}\label{Ck} For every $k\ge 0$, 
\begin{itemize} \item [(i)] $A_{k+1}=\cup_{i=1}^n \phi_i (A_k)\subset A_k$ ;
\item [(ii)] Let $W_1=I \setminus\cap_{k\ge 0} A_k$, then $W_1=\bar{I}$ almost surely. 
\end{itemize}
\end{lemma}
\begin{proof} The equality in claim (i) follows from the following equalities:
$$A_{k+1}=\bigcup_{g\in\mathscr{C}_{k+1}}g(\bar{I})=\bigcup_{i=1}^n\bigcup_{h\in\mathscr{C}_k}\phi_i\big(h(\bar{I})\big)
=\bigcup_{i=1}^n\phi_i\Big(\bigcup_{h\in\mathscr{C}_k}h(\bar{I})\Big)=\bigcup_{i=1}^n \phi_i(A_k).
$$
It follows easily from (\ref{ckak}) that
$\mathscr{C}_{k+1}=\{h\circ\phi_{i} \mid 1\le i\le n, h \in \mathscr{C}_{k} \}$, thus
$$
A_{k+1}=\bigcup_{h\in\mathscr{C}_k} \bigcup_{i=1}^n h\big(\phi_i(\bar{I})\big) \subset \bigcup_{h\in\mathscr{C}_k} \bigcup_{i=1}^n h\big(\bar{I}\big) =  \bigcup_{h\in\mathscr{C}_k} h\big(\bar{I}\big) =A_k
$$
which concludes the proof of item (i). The proof of  claim (ii) follows from the change of variables formula for Lipschitz maps together with claim (i) and equation (\ref{rho}),
$$
\mu(A_{k+1})\le \sum_{i=1}^n \mu\left(\phi_i(A_k)\right)\le\sum_{i=1}^n \int_{A_k} \vert D\phi_i\vert\,{\rm d}\mu=\int_{A_k} \bigg(\sum_{i=1}^n \vert D\phi_i\vert\bigg)\,{\rm d}\mu\le\rho \mu(A_k).
$$ 
Therefore $\mu(A_{k})\le \rho^{k}$, for every $k\ge 0$, thus $\mu(\cap_{k\ge 0} A_k)=0$ which proves item (ii).

\end{proof}

If $\Phi$ is a highly contractive IFS  its atractor, $ \cap_{k\ge 0} A_k$, is a null measure set, by Lemma \ref{Ck},  item (ii). Using   \cite[Theorem 3.1]{CGMU}, one obtains that for  every point $(x_1,\ldots,x_{n-1}) \in \Omega_{n-1}\cap W_{1}^{n-1}$,  the map $f_{x_1,\ldots, x_{n-1}}$ has finitely many periodic orbits and is asymptotically periodic. However, the claim of Theorem \ref{main} is stronger and holds for any contractive IFS.

\begin{lemma}\label{st2} 
There exists a full set $W_2\subset I$ such that 
$h^{-1}(\{x\})$ is a finite set for every $x\in W_2$ and $h\in\cup_{k\ge 0} \mathscr{C}_k$.
\end{lemma}
\begin{proof} 
It is proved in \cite{JW} that if $h:\bar{I}\to\bar{I}$ is a Lipschitz map, then $h^{-1}(\{x\})$ is finite for almost every $x\in\bar{I}$.
The lemma follows immediately  from the fact that $\mathscr{C}_k$ is a finite set.
\end{proof}

Hereafter, let $W_1$ and $W_2$ be as in Lemma \ref{Ck},  item (ii), and Lemma \ref{st2}. Set
\begin{equation}\label{WW}
W_{\Phi}=W_1\cap W_2,\,\,\textrm{then}\,\, W_{\Phi}=I\,\,\textrm{almost surely}.
\end{equation}

 \begin{proposition}\label{mainar} 
For each $x \in W_{\Phi}$, $\displaystyle \bigcup_{k\ge 0}\bigcup_{h\in\mathscr{C}_k} h^{-1}(\{x\})$ is a finite subset of $ \bar{I} \; \big\backslash \bigcap_{k\ge 0} A_k$.
 \end{proposition}
\begin{proof} 
Let $x\in W_{\Phi}$. Assume by contradiction that $\cup_{k\ge 0}\cup_{h\in\mathscr{C}_k} h^{-1}(\{x\})$ is an infinite set. By \mbox{Lemma \ref{st2}}, for every $k\ge 0$, the set $\bigcup_{h\in\mathscr{C}_k} h^{-1}(\{x\})$ is finite. Therefore, for infinitely many $k\ge 0$, the set $\bigcup_{h\in\mathscr{C}_k} h^{-1}(\{x\})$ is nonempty and $x \in A_k$. By item (i) of Lemma \ref{Ck}, $x\in \cap_{k\ge 0} A_k$, which contradicts the fact that $x\in W_1$. This proves the first claim.

Let $y\in \bigcup_{k\ge 0}\bigcup_{h\in\mathscr{C}_k} h^{-1}(\{x\})$, then there exist $\ell\ge 0$ and $h_\ell \in \mathscr{C}_\ell$ such that $x=h_\ell(y)$. Assume by contradiction that $y\in \cap_{k\ge 0} A_k$. Then  $x\in h_\ell ( A_k) \subset \cup_{h\in\mathscr{C}_\ell} h(A_k)= A_{\ell+k}$ for every $k\geq 0$, implying that $x\in \cap_{k\ge \ell} A_k = \cap_{k\ge 0} A_k$, which is a contradiction. This proves the second claim. \end{proof}

 \begin{theorem}\label{Qifinite}  
Let $(x_1,\ldots,x_{n-1})\in \Omega_{n-1}\cap W_{\Phi}^{n-1}$ and  $f=f_{\phi_1,\ldots,\phi_n,x_1,\ldots,x_{n-1}}$, then the set
\begin{equation}\label{tap}
Q=\bigcup_{i=1}^{n-1}\bigcup_{k\ge0} f^{-k}(\{x_i\})
\end{equation}
is finite. Moreover, $Q\subset I\setminus\cap_{k\ge 0} A_k$.
\end{theorem}
\begin{proof} 
Let $(x_1,\ldots,x_{n-1})\in \Omega_{n-1}\cap W_{\Phi}^{n-1}$.   
By Proposition \ref{mainar}, the set $\cup_{k\ge 0}\cup_{h\in\mathscr{C}_k} h^{-1}(\{x_i\})$ is finite for every $1\le i\le n-1$. 
Hence, as 
$$
\bigcup_{k\ge 0}f^{-k}(\{x_i\})\subset\bigcup_{k\ge 0}\bigcup_{h\in\mathscr{C}_k} h^{-1}(\{x_i\}), \quad 1\le i\le n-1,
$$
we have that $Q$ is also a finite set. Moreover, $Q\subset I\setminus \cap_{k\ge 0} A_k$ by Proposition \ref{mainar}. 
 \end{proof}
 
 Next corollary assures that  Theorem \ref{Qifinite} holds if the partition $[x_{0},x_1)$, \ldots, $[x_{n-1},x_n)$ in (\ref{mapf}) is replaced by any partition $I_1,\ldots,I_n$ with each interval $I_i$ having endpoints 
$x_{i-1}$ and $x_i$.
 
 \begin{corollary}\label{Rem1} 
Let $f=f_{\phi_1,\ldots,\phi_n,x_1,\ldots,x_{n-1}}$ be as in Theorem \ref{Qifinite}. Let  $\tilde{f}:I\to I$ be a map having the following properties:
 \begin{itemize}
 \item [(P1)] $\tilde{f}(x)=f(x)$ for every $x\in (0,1)\setminus \{x_1,\ldots,x_{n-1}\}$;
 \item [(P2)] $\tilde{f}(x_i)\in \{\lim_{x\to x_{i}-} f(x),\lim_{x\to x_{i}+} f(x)\}$ for every $1\le i\le n-1$.
 \end{itemize}
 Then the set $\tilde{Q}=\bigcup_{i=1}^{n-1}\bigcup_{k\ge0} \tilde{f}^{-k}(\{x_i\})$ is finite.
 \end{corollary}
 \begin{proof} 
The definition of $f$ given by (\ref{mapf}) together with
 the properties (P1) and (P2) assure that there exists a partition of $I$ into $n$ intervals 
 $I_1,\ldots,I_n$ such that, for every $1\le i\le n$, the interval $I_i$ has endpoints $x_{i-1}$ and $x_i$ and
 $\tilde{f}\vert_{I_i}=\phi_i\vert_{I_i}$. In particular, we have that
 $\tilde{Q}\subset\bigcup_{i=1}^{n-1} \bigcup_{k\ge 0}\bigcup_{h\in\mathscr{C}_k}h^{-1}(\{x_i\})$ which is a finite set by Proposition \ref{mainar}.
 \end{proof}
 
 We remark that, in the next  definition  and in Lemma \ref{tmt},  the IFS is not assumed to be highly contractive.
    
  \begin{definition}\label{partition} 
Let $(x_1,\ldots,x_{n-1})\in\Omega_{n-1}$ and $f=f_{\phi_1,\ldots,\phi_n,x_1,\ldots,x_{n-1}}$ be such that the set $Q$ defined in (\ref{tap}) is finite. The  collection $\mathscr{P}=\{J_{\ell}\}_{\ell=1}^m$ of all connected components of $(0,1)\setminus Q$ is called the {\it invariant quasi-partition} of $f$. In this case,  we say that $f$ {\it has an  invariant quasi-partition}.
  \end{definition}
  
The existence of an invariant quasi-partition plays a fundamental role in  this article. 
   
  \begin{lemma}\label{tmt} 
Let $f=f_{\phi_1,\ldots,\phi_n,x_1,\ldots,x_{n-1}}$ and $\mathscr{P}=\{J_{\ell}\}_{\ell=1}^m$ 
   be as in Definition \ref{partition}, then for every interval $J\in \mathscr{P}$ 
  there exists an interval $J' \in \mathscr{P}$ such that $f(J)\subset J'$.
 \end{lemma}
 \begin{proof} Assume the lemma is false, then there exists $J \in \mathscr{P}$ such that $f(J)\cap  Q \neq\emptyset$. Hence, 
$J \cap f^{-1}(Q)  \neq\emptyset$.
However, $f^{-1}(Q)\subset Q$ implying that  $J \cap Q  \neq\emptyset$ which  contradicts the definition of $\mathscr{P}$.
 \end{proof}

As the next lemma shows, the existence of an invariant quasi-partition $\mathscr{P}$ implies the following weaker notion of periodicity. Let $d:I\to \{1,\ldots,n\}$ be the piecewise constant function defined by $d(x)=i$ if $x\in I_i$. The {\it itinerary} of the point $x\in I$ is the sequence of digits $d_0,d_1,d_2,\ldots$ defined by $d_k=d\left(f^k(x)\right)$. 
We say that {\it the itineraries of $f$ are eventually periodic} if the sequence $d_0,d_1,d_2,\ldots$ is eventually periodic for every $x\in I$. 

\begin{lemma}\label{ipartition} 
Let $(x_1,\ldots,x_{n-1})\in \Omega_{n-1}\cap W_{\Phi}^{n-1}$, then  all itineraries of $f=f_{\phi_1,\ldots,\phi_n,x_1,\ldots,x_{n-1}}$ are eventually periodic.
\end{lemma}
\begin{proof} 
By Theorem \ref{Qifinite}, $Q$ is finite, thus $f$ has an  invariant quasi-partition $\mathscr{P}=\{J_{\ell}\}_{\ell=1}^m$ as in Definition \ref{partition}. By Lemma \ref{tmt}, there exists a map
$\tau:\{1,\ldots,m\}\to \{1,\ldots,m\}$ such that $f\left(J_\ell \right)\subset J_{\tau(\ell)}$ for every $1\le \ell \le m$. Let $1 \le \ell_0 \le m$ and $\{\ell_k\}_{k=0}^\infty$ be the sequence defined recursively by  
$\ell_{k+1}=\tau (\ell_k)$ for every  $k\ge 0$.
It is elementary that  the sequence $\{\ell_k\}_{k=0}^\infty$ is eventually periodic.
We have that $x_i\in Q$ (see (\ref{tap})) for every $1\le i \le n-1$, therefore, by (P1), there exists a unique map $\eta: \{1,\ldots ,m\} \to \{1,\ldots ,n\}$ satisfying $ J_{\ell} \subset I_{\eta(\ell)}$ for every $1 \le \ell \le m$,
Hence, the sequence $\{\eta(\ell_k)\}_{k=0}^\infty$ is eventually periodic and,
by definition, so is the itinerary of any $x \in  J_{\ell_0}$.

Now let $x\in \{0\}\cup Q$. If $\{x,f(x),f^2(x),\ldots\}\subset Q$, then the orbit of $x$ is finite and so its itinerary is eventually  periodic. Otherwise, there exist  $1\le \ell_0 \le m$ and $k\ge 1$ such that $f^k(x)\in J_{\ell_0}$. By the above, the itinerary of $f^k(x)$ is eventually periodic and so is that of $x$. This proves the lemma.
\end{proof}

 \begin{proof}[Proof of Theorem \ref{pr1}]
 Let $(x_1,\ldots,x_{n-1})\in\Omega_{n-1}\cap W_{\Phi}^{n-1}$ and $f=f_{\phi_1,\ldots,\phi_n,x_1,\ldots,x_{n-1}}$.
 By Theorem \ref{Qifinite}, $Q$ is finite, thus $f$ has an  invariant quasi-partition $\mathscr{P}=\{J_{\ell}\}_{\ell=1}^m$ as in Definition \ref{partition}. Let $1\le \ell_0 \le m$ and $x\in  J_{\ell_0}$. In the proof of Lemma \ref{ipartition}, it was proved that the itinerary of $x$ in  $\mathscr{P}$, $\{\ell_k\}_{k=0}^{\infty}$,   is eventualy periodic. Therefore there exist an integer $s \ge 0$ and an even integer $p\ge 2$  such that 
  $\ell_{s}=\ell_{s+p}$. As $\mathscr{P}$ is invariant under $f$,
  $f^{p}(J_{\ell_s})\subset  J_{\ell_{s+p}}=J_{\ell_s}$. Hence, if $J_{\ell_s}=(c,d)$, there exist $0\le c\le c'\le d'\le d\le 1$ such that $\overline{f^{p}\left((c,d)\right)}= [c',d']$. We claim that  $c'>c$.
Assume by contradiction that $c'=c$. As $f^{p}\vert_{(c,d)}$ is a nondecreasing contractive map, there exist $0\le \eta<\epsilon$ such that
 $\overline{f^{p}((c, c+\epsilon))}= [c, c+\eta]$. By induction, for every integer $k\ge 1$, there exist
 $0\le \eta_k<\epsilon_k$ such that $\overline{f^{kp}((c,c+\epsilon_k))}= [c, c+\eta_k]$. Hence,
 $$ c\in\bigcap_{k\ge 1}\bigcup_{h\in\mathscr{C}_{kp}} h(I)=\bigcap_{k\ge 1} A_{kp}=
 \bigcap_{k\ge 0} A_k.
 $$
This contradicts the fact that $c\in\partial J_{\ell_s}\subset \{0,1\}\cup Q\subset I\setminus\cap_{k\ge 0} A_k$ (see \mbox{Theorem \ref{Qifinite}}). We conclude therefore that $c'>c$. Analogously, $d'<d$.
In this way, there exists $\xi>0$ such that
$f^{p}\left((c,d)\right)\subset (c+\xi,d-\xi)$.  As $f^{p}\vert_{(a,b)}$ is a continuous contraction, $f^{p}$ has a unique fixed point $z\in (c,d)$.
Notice that $O_f(z)$ is a periodic orbit. Moreover, it is clear that $\omega(x)=O_f(z)$ for every $x\in J_{\ell_0}$.

Now let $x\in I \setminus\cup_{\ell=1}^m J_{\ell}=\{0\} \cup Q$.
By the proof of Lemma \ref{ipartition}, either  $O_f(x)$ is contained in the finite set $I \setminus\cup_{\ell=1}^m J_{\ell}$ (and thus is finite)
or there exists  $k\geq 1$ such that $f^k(x)\in\cup_{\ell=1}^m J_{\ell}$. By the above, 
in either case, $\omega_f(x)$ is a periodic orbit.
\end{proof}

 \section{Iterated Function Systems}\label{epi}
 
 In this section we prove the following improvement of Theorem \ref{pr1}.
   
\begin{theorem}\label{mr1}
Let $\Phi=\{\phi_1,\ldots,\phi_n\}$ be an IFS formed by $\kappa$-Lipschitz functions, with $0\le \kappa <\frac12$, defined on $I=[0,1)$, then there exists a full set 
$W_{\Phi}\subset I$ such that for every $(x_1,\ldots,x_{n-1})\in\Omega_{n-1}(I)\cap W_{\Phi}^{n-1}$, the  PC $f_{\phi_1,\ldots,\phi_n,x_1,\ldots,x_{n-1}}$ has an  invariant quasi-partition and is asymptotically periodic.
\end{theorem}

Theorem \ref{mr1} will be deduced from Theorem \ref{pr1} in the following way.
First, we show that  the IFS $\Phi$ can be  locally replaced by a highly contractive IFS $\Upsilon$ and then that the local substitution suffices to prove  \mbox{Theorem \ref{mr1}}.

Hereafter, let  $(x_1,\ldots,x_{n-1})\in \Omega_{n-1}$ be fixed. Set  $x_0=0$, $x_n=1$,
\begin{equation}\label{delta} 
 \delta=\min_{1\le i\le n}  \frac{x_{i}-x_{i-1}}{3}\; \textrm{and} \;\;
V(x_1,\ldots,x_{n-1})=\{(y_1,\ldots,y_{n-1})\in\Omega_{n-1}: \vert y_i-x_i\vert< \delta,\forall i\}. 
\end{equation}

In what follows, let $0\le \kappa<\frac12$ and $\phi_1,\ldots,\phi_n:[0,1]\to (0,1)$ be $\kappa$-Lipschitz contractions.  Let
$\Upsilon=\{\varphi_1,\ldots,\varphi_n\}$ be the IFS defined by 
$$
\varphi_1(x)=\begin{cases} \phi_1(x)\,\,\phantom{aaaa}\textrm{, if}\,\,x\in[0,x_{1}+\delta] \\ \phi_1(x_1+\delta)\,\,\phantom{}\textrm{, if}\,\,x\in[x_{1}+\delta,1] 
\end{cases},\; \varphi_n(x)=\begin{cases}\phi_n(x_{n-1}-\delta)\,\,\phantom{}\textrm{, if}\,\,x\in[0,x_{n-1}-\delta] \\   \phi_n(x)\,\,\phantom{aaaaaa}\textrm{, if}\,\,x\in[x_{n-1}-\delta,1] 
\end{cases}
$$
$$
\varphi_i(x)=\begin{cases} \phi_i(x_{i-1}-\delta)\,\,\textrm{, if}\,\,x\in [0,x_{i-1}-\delta]\\ \phi_i(x)\,\,\phantom{aaaaaa}\textrm{, if}\,\,x\in[x_{i-1}-\delta,x_{i}+\delta] \\ \phi_i(x_i+\delta)\,\,\phantom{aa}\textrm{, if}\,\,x\in [x_{i}+\delta,1]
\end{cases},\quad 2\le i \le n-1.
$$
\begin{figure}[h]
\begin{center}
\includegraphics[width=\linewidth]{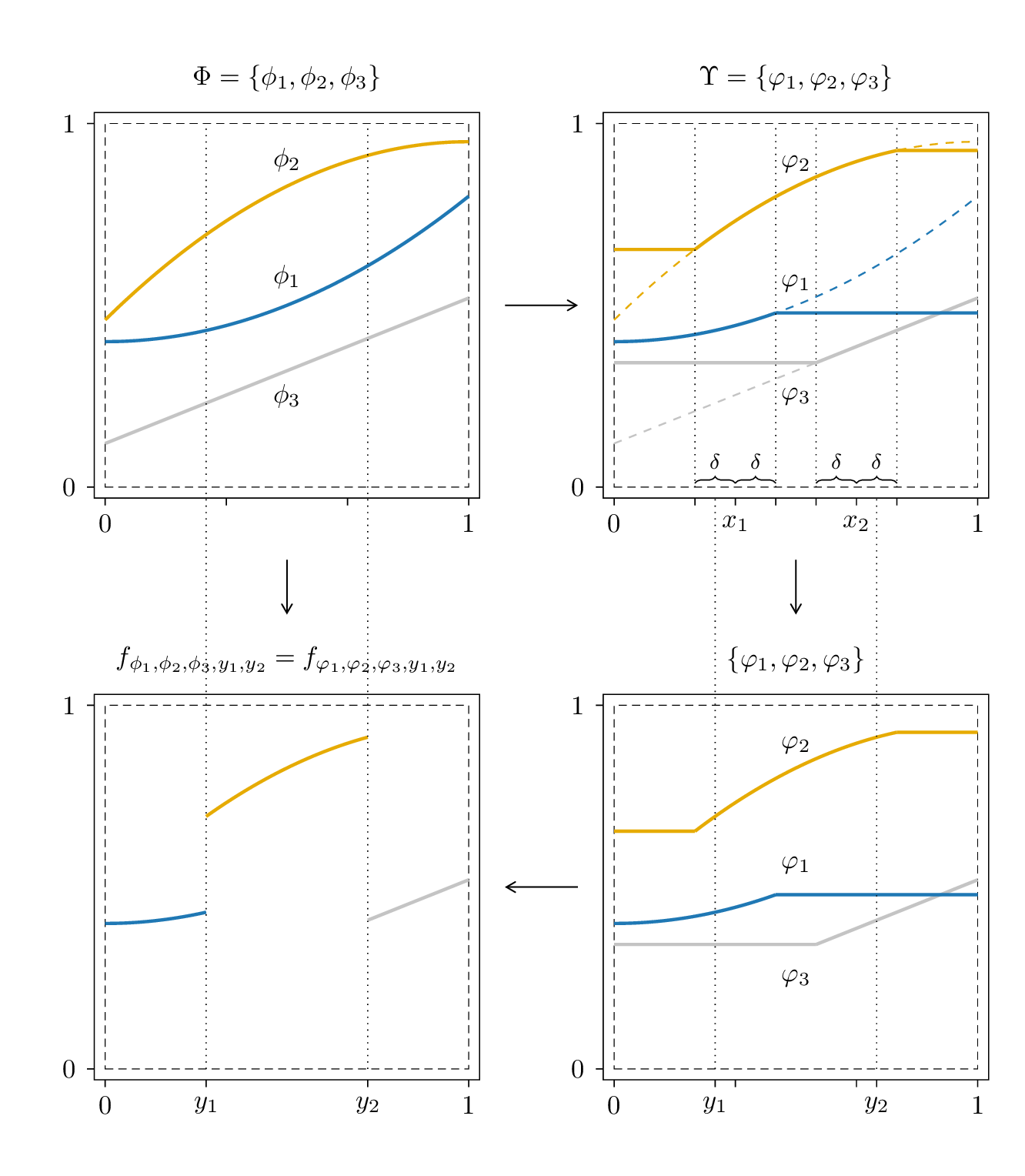}
\caption{Relations between the IFS $\{\phi_1,\phi_2,\phi_3\}$ and $\{\varphi_1,\varphi_2,\varphi_3\}$.}\label{fig:ifss}
\end{center}
\end{figure}

A scheme illustrating the construction of the
IFS $\{\varphi_1,\ldots,\varphi_n\}$ from the IFS $\{\phi_1,\ldots,\phi_n\}$ is shown in Figure \ref{fig:ifss}.
 
 \begin{lemma}\label{ispecial} 
The IFS $\Upsilon=\{\varphi_1,\ldots,\varphi_n\}$ is highly contractive.
 \end{lemma}
\begin{proof} 
It is clear that each $\varphi_i:[0,1]\to (0,1)$ is a Lipschitz contraction, therefore for almost every $x\in I$, $D\varphi_i(x)$ exists, and, by the definition of $\varphi_i$, 
$$
\vert D\varphi_1(x)\vert+\cdots+\vert D\varphi_n(x)\vert\le \max_{1\le i\le n-1} \left(\vert D\phi_{i}(x)\vert+\vert D\phi_{i+1}(x)\vert\right)\leq 2\kappa<1.
$$
Therefore, $\{\varphi_1,\ldots,\varphi_n\}$ is highly contractive which concludes the proof.
\end{proof}

In what follows, let $V=V(x_1,\ldots,x_{n-1})$ be as in (\ref{delta}).

\begin{lemma}\label{replace} 
For every $(y_1,\ldots,y_{n-1})\in V$,  
$$
f_{\varphi_1,\ldots,\varphi_n,y_1,\ldots,y_{n-1}}= f_{\phi_1,\ldots,\phi_n,y_1,\ldots,y_{n-1}}.
$$
\end{lemma}
 \begin{proof} 
Let $(y_1,\ldots,y_{n-1})\in V$, $y_0=0$ and $y_n=1$. Set $K_i=[y_{i-1},y_i)$ for  $1\le i\le {n}$.  Since $\vert y_i-x_i\vert<\delta$, we have that
$$
K_1\subset [0,x_1+\delta],  K_n\subset [x_{n-1}-\delta,1]\,\,\textrm{and}\,\, K_i\subset [x_{i-1}-\delta,x_i+\delta],\quad 2\le i\le n-1.
$$ 
This together with the definition of $\varphi_i$ yields $\varphi_i\vert_{K_i}=\phi_i\vert_{K_i}$, for every $1\le i\le n$.
 \end{proof}
 
 Now we will apply the results of Section \ref{IFS} to the highly contractive IFS $\Upsilon=\{\varphi_1,\ldots,\varphi_n\}$ (see Lemma \ref{ispecial}). With respect to such IFS, 
 let $W_{\Upsilon}\subset I$ be the full set defined in equation (\ref{WW}). 
 Notice that all the claims in Section \ref{IFS} hold true for the IFS $\Upsilon$. In particular, we have that
 $V\cap W_{\Upsilon}^{n-1}=V$ almost surely. 
 
 \begin{theorem}\label{lmim} 
For every
 $(y_1,\ldots,y_{n-1})\in V\cap W_{\Upsilon}^{n-1}$, $f= f_{\phi_1,\ldots,\phi_n,y_1,\ldots,y_{n-1}}$ has an invariant quasi-partition and is asymptotically periodic.
 \end{theorem}
 \begin{proof} 
Let $(y_1,\ldots,y_{n-1})\in V\cap W_{\Upsilon}^{n-1}$.
 By Lemma \ref{replace}, $f=f_{\varphi_1,\ldots,\varphi_n,y_1,\ldots,y_{n-1}}$. By \mbox{Lemma \ref{ispecial}}, we have that $\Upsilon$
 fulfills the hypotheses of Theorems \ref{pr1} and  \ref{Qifinite}. Hence, $f$ has an  invariant quasi-partition and is asymptotically periodic.  
 \end{proof}

We stress that, in the previous results, $\delta=\min_{1\le i\le n} (x_{i}-x_{i-1})/{3}$,
the set $V$ and the IFS $\Upsilon$ depend on the point $(x_1,\ldots,x_{n-1})$. For this reason,
in the next proof, we replace $V$ and $W_{\Upsilon}$ by $V(x_1,\ldots,x_{n-1})$ and $W(x_1,\ldots,x_{n-1})$, respectively.
\medskip

\begin{proof}[Proof of Theorem \ref{mr1}]
Set $\Phi=\{\phi_1,\ldots,\phi_n\}$.
 First we show that
\begin{equation}\label{grande}
\bigcup_{(z_1,\ldots,z_{n-1}) \in \Omega_{n-1}\cap \mathbb{Q}^{n-1}} V(z_1,\ldots,z_{n-1}) = \Omega_{n-1}.
\end{equation}
 Let $(x_1,\ldots,x_{n-1}) \in \Omega_{n-1}$ and  $\delta_0=\delta(x_1,\ldots,x_{n-1})$. Let  $(z_1,\ldots,z_{n-1}) \in \Omega_{n-1}\cap \mathbb{Q}^{n-1}$ be such that  $\vert z_i- x_i \vert < \frac 12 {\delta_0}$ for every $1\le i \le n-1$.
 Set $z_0=0$ and $z_n=1$.
We have that for every $1\le j \le n$, $z_{j} - z_{j-1}> (x_j-\frac 12 {\delta_0}) - (x_{j-1}+\frac 12 {\delta_0})\ge 3 {\delta_0} - {\delta_0} =2 {\delta_0}$. 
In this way, $\delta(z_1,\ldots,z_{n-1})> \frac 23 \delta_0 > \frac 12 \delta_0$, thus $(x_1,\ldots,x_{n-1}) \in V(z_1,\ldots,z_{n-1})$. 
This proves (\ref{grande}).
 
Let $(z_1,\ldots,z_{n-1}) \in \Omega_{n-1}\cap \mathbb{Q}^{n-1}$ and let $W(z_1,\ldots,z_{n-1})$  be the full set in $I$ defined  by (\ref{WW}).  By Theorem \ref{lmim}, for every 
$(y_1,\ldots,y_{n-1}) \in V(z_1,\ldots,z_{n-1}) \cap W(z_1,\ldots,z_{n-1})^{n-1}$,
the map $ f_{\phi_1,\ldots,\phi_n,y_1,\ldots,y_{n-1}}$ has an  invariant quasi-partition and is asymptotically periodic. The denumerable intersection
\begin{equation}\label{w}
W_{\Phi}=\bigcap_{(z_1,\ldots,z_{n-1}) \in \Omega_{n-1}\cap \mathbb{Q}^{n-1}} W(z_1,\ldots,z_{n-1})
\end{equation}
is a full set
 and, for every $(y_1,\ldots,y_{n-1}) \in V(z_1,\ldots,z_{n-1})\cap W_{\Phi}^{n-1}$, the map $ f_{\phi_1,\ldots,\phi_n,y_1,\ldots,y_{n-1}}$ has an  invariant quasi-partition and is asymptotically periodic.
This together with (\ref{grande}) concludes the proof.
\end{proof}

\section{Asymptotic periodicity: the general case}\label{gc}

In this section we prove the following improvement of Theorem \ref{mr1}.

\begin{theorem}\label{toconclude2} Let $\Phi=\{\phi_1,\ldots,\phi_n\}$ be an IFS formed by $\rho$-Lipschitz functions defined on $I=[0,1)$ with $0\le\rho<1$, then there exists a full set 
$W_{0}\subset I$ such that for every $(x_1,\ldots,x_{n-1})\in\Omega_{n-1}(I)\cap W_{0}^{n-1}$, the  PC $f_{\phi_1,\ldots,\phi_n,x_1,\ldots,x_{n-1}}$ has an  invariant quasi-partition and is asymptotically periodic.
\end{theorem}

Throughout this section, let $0\le\rho<1$ and $\phi_1,\ldots,\phi_n:[0,1]\to (0,1)$ be $\rho$-Lipschitz contractions. Let $k\ge 1$ be such that $\rho^k<\frac12$. 
 By the Chain rule for Lipschitz maps, $\mathscr{C}_k=\mathscr{C}_k(\phi_1,\ldots,\phi_n)$ \big(see (\ref{ckak})\big) is a collection of at most $n^k$  $\rho^k$-Lipschitz contractions.
 
 For each $r\ge 2$, let $\mathscr{I}_r$ denote the collection of all IFS $\{\psi_1,\ldots,\psi_r\}$, where
 each $\psi_j$, $1\le j\le r$, belongs to  $\mathscr{C}_k(\phi_1,\ldots,\phi_n)$. The collection $\mathscr{I}_r$ consists of at most  $\dfrac{n^k!}{(n^k-r)!}$ IFS. Notice that in an IFS, the order in which the maps are listed matters. The fact that $\rho^k<\frac12$ implies that any IFS in $\cup_{r\ge 2}\mathscr{I}_r$ satisfies the hypothesis of Theorem \ref{mr1}. 
 
 In the statement of Theorem \ref{mr1}, the set $W_{\Phi}$ depends  on the IFS
 $\Phi=\{\phi_1,\ldots,\phi_n\}$. In the next corollary, the set $W$ does not depend on the IFS $\Psi=\{\psi_1,\ldots,\psi_r\}$, provided $\Psi$ is chosen within the denumerable collection $\cup_{r\ge 2}\mathscr{I}_r$.
    
 \begin{corollary}\label{cmr1} There exists a full set $W\subset I$ such that for every
 $r\ge 2$, $\{\psi_1,\ldots,\psi_r\}\in\mathscr{I}_r$ and $(y_1,\ldots,y_{r-1})\in \Omega_{r-1}\cap W^{r-1}$, the
 $r$-interval PC $f_{\psi_1,\ldots,\psi_r,y_1,\ldots,y_{r-1}}$ has an  invariant quasi-partition and is asymptotically periodic.
\end{corollary}
\begin{proof}  
Let $\mathscr{I}=\cup_{r\ge 2}\mathscr{I}_r$. By \mbox{Theorem \ref{mr1}},  for each
IFS $\Psi=\{\psi_1,\ldots,\psi_r\}\in\mathscr{I}$, there exists a full set $W_{\Psi}\subset I$ such that the following holds: 
for every $(y_1,\ldots,y_{r-1})\in \Omega_{r-1}\cap W_{\Psi}^{r-1}$, the $r$-interval PC
$g=f_{\psi_1,\ldots,\psi_{r},y_1,\ldots,y_{r-1}}$ has an  invariant quasi-partition and is asymptotically periodic. The proof is concluded by taking $W=\cap_{\Psi\in\mathscr{I}} W_{\Psi}$. Since $\mathscr{I}$ is denumerable, we have that $W$ is a full subset of $I$.
\end{proof}

 \begin{corollary}\label{sml} 
 Let $\{\psi_1,\ldots,\psi_r\}\in\mathscr{I}_r$ and $(y_1,\ldots,y_{r-1})\in \Omega_{r-1}\cap W^{r-1}$, where  $r\ge 2$,  satisfy the hypothesis of Corollary \ref{cmr1}.
Let $g=f_{\psi_1,\ldots,\psi_{r},y_1,\ldots,y_{r-1}}$ and   $\tilde{g}:I \to I$ be any map having the following properties:
\begin{itemize}
 \item [(P1)] $\tilde{g}(y)=g(y)$ for every $y\in (0,1)\setminus \{y_1,\ldots,y_{r-1}\}$;
 \item [(P2)] $\tilde{g}(y_j)\in \{\lim_{y\to y_{j}-} g(y),\lim_{y\to y_{j}+} g(y)\}$ for every $1\le j\le r-1$.
 \end{itemize}
 Then the map $\tilde{g}$ has an invariant quasi-partition and is asymptotically periodic.
\end{corollary} 
\begin{proof} Let $x\in I$. If $O_{\tilde{g}}(x)\subset  \{0\}\cup\{y_1,\ldots,y_{r-1}\}$, then $O_{\tilde{g}}(x)$ is finite. Otherwise, there exists $\ell\ge 0$ such that $O_{\tilde{g}}\left({\tilde g}^{\ell}(x)\right)\subset (0,1)\setminus \{y_1,\ldots,y_{r-1}\}$. In this case, by (P2), we have that $O_{\tilde{g}}\left(\tilde{g}^{\ell}(x)\right)=O_g\left(\tilde{g}^{\ell}(x)\right)$, which is finite by \mbox{Corollary \ref{cmr1}}. This proves that $\tilde{g}$ is asymptotically periodic.

It remains to be shown that the set $\tilde{Q}=\cup_{j=1}^{r-1}\cup_{k\ge 0} {\tilde{g}}^{-k}(\{y_j\})$ is finite. By proceeding as in the proof of Corollary \ref{Rem1}, it can be proved that the claims of
Lemma \ref{replace}, \mbox{Theorems \ref{lmim}} and \ref{mr1} and therefore Corollary \ref{cmr1}
 hold if we replace in (\ref{mapf}) the partition $[x_{0},x_1)$, \ldots, $[x_{n-1},x_n)$ by any partition $I_1,\ldots,I_n$  where each interval $I_i$ has endpoints 
$x_{i-1}$ and $x_i$. This means that in Corollary \ref{cmr1} we can replace the map $g$ by the map $\tilde{g}$ and conclude that the set $\tilde{Q}$ is finite. Hence, $\tilde{g}$ has an invariant quasi-partition. 
 \end{proof} 

Corollary \ref{cmr1} and Corollary \ref{sml} will be used later on this section. Now let us come back to the original IFS $\{\phi_1,\ldots,\phi_n\}$.
 
We denote by $\Omega_{n-1}'$  the set  
\begin{equation}\label{ind}
\Omega_{n-1}'=\Omega  \setminus \bigcup_{i= 0}^{n-1} \bigcup_{j= 1}^{n-1}\bigcup_{\ell\ge 1}\bigcup_{h \in \mathscr{C}_\ell} 
\{(x_1,\ldots,x_{n-1})\in\Omega_{n-1} : x_j=h(x_i)\},
\end{equation}
which will be used in the forthcoming results.

\begin{lemma}\label{inde} 
$\Omega_{n-1}'=\Omega_{n-1}$ almost surely.
\end{lemma} 
\begin{proof} 
There are only denumerably many sets of the form 
 $\{(x_1,\ldots,x_{n-1})\in\Omega_{n-1}: x_j=h(x_i)\}$, where 
$0\le i\le n-1$, $1\le j\le n-1$ and $h\in \bigcup_{\ell\ge 1} \mathscr{C}_{\ell}$.
Being the graph of a function, each such set is a null set.
Therefore,  $\Omega_{n-1}'$ equals $\Omega_{n-1}$ up to a null set. 
\end{proof}

\begin{lemma}\label{sta} 
Let $(x_1,\ldots,x_{n-1})\in\Omega_{n-1}'$  and $f=f_{\phi_1,\ldots,\phi_n,x_1,\ldots,x_{n-1}}$. Let   $\gamma$ be a periodic orbit of $f$, then there exists a neighborhood $U\subset I$ of $\gamma$ such that $f(U)\subset U$ and $\gamma=\cap_{\ell\ge 0} f^{\ell}(U)$. Moreover, $\omega_f(x)=\gamma$ for every $x\in U$.
\end{lemma} 
\begin{proof} 
Let $\gamma$ be a periodic orbit of $f$. 
As $(x_1,\ldots,x_{n-1})\in\Omega_{n-1}'$ and
$f(I)\subset (0,1)$, we have that
$\gamma\cap \{x_0,\ldots,x_{n-1}\}=\emptyset$. Let $\epsilon=\frac 12 \min\{\vert x-x_i\vert: x\in \gamma, 0\le i\le n\}$ and set $U:=\cup_{x\in \gamma}\left( x-\epsilon,x+\epsilon\right)$, in particular $U\subset I\setminus \{x_0,\ldots,x_{n-1}\}$. This together with the fact that $f\vert_{[x_{i-1},x_i)}$ is a Lipschitz contraction implies that $f(U)\subset U$, thus $\gamma=\cap_{\ell\ge 0} f^{\ell}(U)$.
\end{proof}

\begin{lemma}\label{oneofu} Let $(x_1,\ldots,x_{n-1})\in\Omega_{n-1}'$  and $f=f_{\phi_1,\ldots,\phi_n,x_1,\ldots,x_{n-1}}$, then $f^k$ is left-continuous or right-continuous at each point of $I$.
\end{lemma}
\begin{proof} Let $y\in I$ and $S_y=\{y,f(y),\ldots,f^{k-1}(y)\}$. The fact that $(x_1,\ldots,x_{n-1})\in\Omega_{n-1}'$  assures that $S_y\cap \{x_0,\ldots,x_{n-1}\}$ is either empty or an one-point set. In the former case,
we have that $f$ is continuous on $S_y$, hence $f^k$ is continuous at $y$. In the latter case, there exists
$y'\in S_y$ such that $f$ is continuous at each point of $S_y\setminus\{ y'\}$ and $f$ is left-continuous or right-continuous at $y'$. Accordingly, $f^k$ is either left-continuous or right-continuous at $y$.
\end{proof}

For the next result, let $W$ be the full set in the statement of Corollary \ref{cmr1}.

\begin{lemma}\label{toconclude1} 
There exists a full  set $W_0\subset W$ such that 
if $(x_1,\ldots,x_{n-1})\in \Omega_{n-1}'\cap W_0^{n-1}$ and $f=f_{\phi_1,\ldots,\phi_n,x_1,\ldots,x_{n-1}}$, then 
$$
f^k\vert_{(y_{j-1},y_j)}=\psi_j\vert_{(y_{j-1},y_j)},\,\,1\le j\le r,
$$
for some $r\ge 2$, $(y_1,\ldots,y_{r-1})\in\Omega_{r-1}\cap W^{r-1}$, and $\psi_1,\ldots, \psi_{r}\in \mathscr{C}_k$. Moreover, $f^k$ is left-continuous or right-continuous at each point of $I$. 
\end{lemma}
\begin{proof} 
Let $M=\Big\{x\in W\,\big | \cup_{\ell=0}^{k-1}\cup_{h\in\mathscr{C_{\ell}}}h^{-1}(\{x\})\not\subset W \Big\}$. Notice that $M\subset \cup_{\ell=0}^{k-1}\cup_{h\in\mathscr{C}_{\ell}} h(I\setminus W)$, where $I\setminus W$ is a null set, therefore $M$ is a null set. 
By Lemma \ref{st2}, there exists a full  set $W_0\subset W\setminus M$ such that 
 $\bigcup_{\ell=0}^{k-1}f^{-\ell}(\{x\})\subset \bigcup_{\ell=0}^{k-1}\bigcup_{h\in\mathscr{C}_{\ell}}h^{-1}(\{x\})$ is a finite subset of $W$ for every $x\in W_0$. 
Now let $(x_1,\ldots,x_{n-1})\in \Omega_{n-1}'\cap W_0^{n-1}$, thus $\bigcup_{\ell=0}^{k-1}f^{-\ell}(\{x_1,\ldots,x_{n-1}\})$ is a finite subset of $W\setminus\{0\}$ and we may list its elements in ascending order $0<y_1<\cdots <y_{r-1}<1$. In this way, $(y_1,\ldots,y_{r-1})\in \Omega_{r-1}\cap W^{r-1}$.
 Let us analyze how $f^k$ acts on the the intervals
 $E_1=(y_0,y_1)$, \ldots, $E_{r}=(y_{r-1},y_r)$. Fix $1\le j\le r$.
 Since $\{y_1,\ldots,y_{r-1}\}=\bigcup_{\ell=0}^{k-1}f^{-\ell}(\{x_1,\ldots,x_{n-1}\})$, we have that
 for each $0\le \ell\le k-1$,  there exists a unique $1\le i_{\ell}\le n$ such that
 $f^{\ell}(E_j)\subset (x_{i_{\ell}-1},x_{i_{\ell}})$. This together with the fact that $f\vert_{[x_{i_{\ell}-1},x_{i_{\ell}})}=\phi_{i_{\ell}}$ yields $f^k\vert_{E_j}=\psi_j\vert_{E_j}$, where
 $\psi_j=\phi_{{i_{k}}}\circ\cdots\circ \phi_{i_{1}}\in\mathscr{C}_k$. The claim that $f^k$ is left-continuous or right-continuous at each point of $I$ follows from \mbox{Lemma \ref{oneofu}}.
 \end{proof}

\begin{proof}[Proof of Theorem \ref{toconclude2}] Let $(x_1,\ldots,x_{n-1})\in \Omega_{n-1}'\cap W^{n-1}_0$ and $f=f_{\phi_1,\ldots,\phi_{n},x_1,\ldots,x_{n-1}}$.
By Lemma \ref{toconclude1}, there exist $r\ge 2$, $(y_1,\ldots,y_{r-1})\in\Omega_{r-1}\cap W^{r-1}$,
and $\psi_1,\ldots,\psi_r\in\mathscr{C}_k$ such that 
\begin{equation}\label{fkk}
f^k\vert_{(y_{j-1},y_j)}=\psi_j\vert_{(y_{j-1},y_j)},\,\,1\le j\le r.
\end{equation}
Let $g=f_{\psi_1,\ldots,\psi_n,y_1,\ldots,y_{r-1}}$ and $\tilde{g}=f^k$. We claim that $\tilde{g}$ satisfies (P1) and (P2) in Corollary \ref{sml}. The property (P1) follows automatically from the equation (\ref{fkk}). The property (P2) follows from (P1) together with the fact that $f^k$ is left-continuous or right-continuous at each point of $I$, as assured by Lemma \ref{toconclude1}. By Corollary \ref{sml}, the map $\tilde{g}=f^k$ has an invariant quasi-partition, that is to say, the set
$$\tilde{Q}=\cup_{j=1}^{r-1}\cup_{s\ge 0} \tilde{g}^{-s}(\{y_j\})=\cup_{j=1}^{r-1}\cup_{s\ge 0} f^{-sk}(\{y_j\})$$
is finite,
implying that the set $Q'=\cup_{j=1}^{r-1}\cup_{s\ge 0} f^{-s}(\{y_j\})$ is finite. By the proof of \mbox{Lemma \ref{toconclude1}}, we have that $\{x_1,\ldots,x_{n-1}\}\subset \{y_1,\ldots,y_{r-1}\}$.
In this way, 
$$Q:=\cup_{i=1}^{n-1}\cup_{s\ge 0} f^{-s}(\{x_i\})\subset \cup_{j=1}^{r-1}\cup_{t\ge 0} f^{-t}(\{y_j\}),$$
and $Q$ is therefore finite. This proves that $f$ has an invariant quasi-partition.

By Corollary \ref{sml}, the map $\tilde{g}=f^k$ is asymptotically periodic. We claim that $f$ is also asymptotically periodic. Let $x\in I$, then there exists a periodic orbit $\gamma_k$ of $f^k$ such that
$\omega_{f^k}(x)=\gamma_k$. Let $p\in \gamma_k$. Notice that $p$ is a periodic point of $f$, thus there exists a periodic orbit $\gamma$ of $f$ that contains $p$ and $\gamma_k$. Let $U$ be a neighborhood of $\gamma$ given by Lemma \ref{sta}. Since $\omega_{f^k}(x)=\gamma_k\subset \gamma$,  there exists
an integer $\eta\ge 1$ such that $f^{\eta k}(x)\in U$. By Lemma \ref{sta}, $\omega_f(x)=\omega_f\left( f^{\eta k}(x)\right)=\gamma$ which proves the claim. Hence, $f$ is asymptotically periodic.
\end{proof}

 \section{An upper bound for the number of periodic orbits}\label{upperb}

  Throughout this section, let $\phi_1,\ldots,\phi_{n}:[0,1]\to (0,1)$ be Lipschitz contractions,  $W_0$ be the full set in the statement of Theorem \ref{toconclude2} and $\Omega_{n-1}'$ be the set defined in (\ref{ind}). Let
  $(x_1,\ldots,x_{n-1})\in\Omega_{n-1}'\cap W_0^{n-1}$ and $f=f_{\phi_1,\ldots,\phi_{n},x_1,\ldots,x_{n-1}}$. By Theorem \ref{toconclude2}, $f$ has an invariant  quasi-partition $\mathscr{P}=\cup_{\ell=1}^m J_{\ell}$ with endpoints in $\{0,1\}\cup\cup_{i=1}^{n-1} Q_i$, where each set $Q_i=\cup_{k\ge 0}f^{-k}(\{x_i\})$ is finite. 
  
  Here we prove the following result.

\begin{theorem}\label{IP1}  The $n$-interval PC $f$ has at most $n$ periodic orbits. 
\end{theorem}

We would like to distinguish some intervals in  $ \mathscr{P}$, first those having $x_0=0$ and $x_n=1$ as endpoints. We denote them by $F_0$  and $G_n$, where  $x_0\in \overline{F_0}$ and $x_n\in \overline{G_n}$. 
For every $1\leq i\leq n-1$, let $F_i=(a,x_i)$ and $G_i=(x_i,b)$ be the two intervals in $\mathscr{P}$ which have $x_i$ as an endpoint. We may have $G_i=F_{i+1}$ for some $1\leq i \leq n-2$. Among the intervals $F_1, G_1,\ldots,F_{n-1},G_{n-1}$, there are at least $n$  and at most $2(n-1)$ pairwise distinct intervals. 
We will prove that among them there are $1\leq r \leq n$  pairwise distinct intervals, say $C_1,\ldots,C_r$, which satisfy the following: for every $J  \in \mathscr{P}$, there exist $k\ge 0$ and $1\le i \le r$ such that  
$$
f^k(J) \subset  C_i.
$$
This implies that the asymptotical behavior of any interval $J  \in \mathscr{P}$ coincides  with the asymptotical behavior of an interval $C_i$.

Let $J,J_1,J_2\in  \mathscr{P}$ and $k \ge 0$. We remark that $ f^{k}(J)\subset J_1\cup J_2$ if, and only if, $ f^{k}(J)\subset J_1$ or $ f^{k}(J)\subset J_2$.

\begin{lemma}\label{P1}
 Let $(a,b)\in \mathscr{P}$ with $a\in Q_i$ and $b\in Q_j$, where $1\leq i, j \leq n-1$ and $i\neq j$. Then there exists $\ell\ge 0$ such that $($at least$)$ one of the following statements holds
 \begin{itemize}
 \item [(i)] $f^{\ell}(F_i)\subset F_j\cup G_j$ or $f^{\ell}(G_i)\subset F_j\cup G_j$;
 \item [(ii)] $f^{\ell}(F_j)\subset F_i\cup G_i$ or $f^{\ell}(G_j)\subset F_i\cup G_i$.
 \end{itemize} 
\end{lemma} 
\begin{proof} The hypotheses that $a\in Q_i$, $b\in Q_j$ and $(x_1,\ldots,x_{n-1}) \in\Omega_{n-1}'$ \big(see (\ref{ind})\big) 
  imply that there exist unique integers $\ell_i,\ell_j\ge 0$ such
 that $f^{\ell_i}(a)=x_i$ and $f^{\ell_j}(b)=x_j$. Moreover, $f^k(a)\not\in \{x_1,\ldots,x_{n-1}\}$ for every $k\neq\ell_i$, and $f^m(b)\not\in \{x_1,\ldots,x_{n-1}\}$ for every $m\neq \ell_j$. Let $J=(a,b)$, then
 $f^{\ell_i}(J)\subset F_i\cup G_i$ and $f^{\ell_j}(J)\subset F_j\cup G_j$. Now it is clear that the claim (i) happens if $\ell_i\le \ell_j$ (then we set $\ell=\ell_j-\ell_i$) and the claim (ii) occurs if  $\ell_i\ge \ell_j$ (then we set $\ell=\ell_i-\ell_j$).
 \end{proof}
 
 \begin{lemma}\label{P0}
Let $J\in  \mathscr{P}$, then there exist $1\leq i\leq n-1$ and  $k \ge 0$ such that $ f^{k}(J)\subset F_i\cup   G_i$.
\end{lemma} 
\begin{proof} It follows by the arguments used in the proof of Lemma \ref{P1}.
\end{proof}

\begin{lemma}\label{P2}
There exists a permutation $i_1,\ldots,i_{n-1}$ of $1,\ldots,n-1$ and intervals
$(a_{k-1},b_k)\in \mathscr{P}$ with
$a_{k-1}\in Q_{i_1}\cup \ldots \cup Q_{i_{k-1}}$ and $b_k\in Q_{i_k}$ for every $2\le k\le n-1$.
\end{lemma} 
\begin{proof} 
Since $Q_1,\ldots , Q_{n-1}$ are pairwise disjoint finite subsets of the interval $(0,1)$, the numbers $y_i=\min Q_i$, $1\le i\le n-1$ are pairwise distinct numbers. Hence, there exists a permutation $i_1,\ldots,i_{n-1}$ of $1,\ldots,n-1$ such that $y_{i_1}<\cdots<y_{i_{n-1}}$. Set $b_k=y_{i_k}\in Q_{i_k}$, $1\le k\le n-1$, therefore
$0\le b_1<b_2<\cdots<b_{n-1}<1$. Notice that 
\begin{equation}\label{simpm}
(0,b_{k})\cap \big(Q_{i_{k}}\cup\cdots\cup Q_{i_{n-1}}\big)=\emptyset\quad\textrm{for every}\quad
2\le k\le n-1.
\end{equation}
For every $2\le k\le n-1$, let $S_k=(0,b_k)\cap \big( Q_1\cup \cdots\cup Q_{n-1}\big)$. Since $b_{k-1}\in  (0,b_k)\cap Q_{i_{k-1}}$, we have that $S_k\neq\emptyset$. Set $a_{k-1}=\max S_k$, then $a_{k-1}<b_k$ and $(a_{k-1},b_k)\in \mathscr{P}$. 
By (\ref{simpm}), $a_{k-1}\in Q_{i_{1}}\cup\cdots\cup Q_{i_{k-1}}$, which concludes the proof.
\end{proof}

Using the permutation $i_1,\ldots,i_{n-1}$ defined in Lemma \ref{P2}, for simplicity, set $F'_k=F_{i_{k}}$ and $G'_k=G_{i_{k}}$, for $1\leq k \leq n-1$.

\begin{corollary}\label{P3}
Let $2\leq k  \leq n-1$, then there exist $1\leq j <k$ and $\ell\geq 0$  such that $($at least$)$ one of the following statements holds$:$\\
$(i)$ $f^{\ell}(F'_j)\subset F'_k \cup G'_k$ or $f^{\ell}(G'_j)\subset F'_{k} \cup G'_{k}$, \\
$(ii)$ $f^{\ell}(F'_k)\subset F'_{j} \cup G'_{j}$ or $f^{\ell}(G'_{k})\subset F'_j \cup G'_j$.
\end{corollary}
\begin{proof} Let $i_1,\ldots,i_{n-1}$ be the permutation  of $1,\ldots,n-1$ given by  Lemma \ref{P2},
then for every $2\le k\le n-1$, there exist $1\le j<k$ and $(a,b)\in\mathscr{P}$ with
$a\in Q_{i_j}$ and $b\in Q_{i_k}$. The interval $(a,b)$ fulfills the hypothesis of Lemma \ref{P1}. 
The proof is finished by making the following substitutions in the claim of Lemma \ref{P1}:
$i=i_j$, $j=i_k$, $F_i=F_{i_j}=F_j'$ and $F_j=F_{i_k}=F_k'$.
\end{proof}

Next we introduce an equivalence relation in the family of intervals $\mathscr{P'}$ listed as
$$ \mathscr{P}'=\{F_1,G_1,\ldots,F_{n-1},G_{n-1}\}=\{F_1',G_1',\ldots,F_{n-1}',G_{n-1}'\}.$$

\begin{definition}
Let  $C_1,C_2 \in  \mathscr{P}'$. We say that $C_1$ and $C_2$ are {\it equivalent} if  there exists $C \in  \mathscr{P}'$ such that $ f^{\ell}(C_1)\cup  f^{k}(C_2)\subset C$ for some  $\ell,k\geq 0$. If $C_1$ and $C_2$ are equivalent, we write $C_1 \equiv C_2$.
\end{definition}

\begin{lemma}\label{P4}
The relation  $\equiv$ is an equivalence relation with at most $n$ equivalence classes.
\end{lemma}
\begin{proof}
It is clear that $\equiv$ is reflexive and symmetric. To prove that $\equiv$ is transitive, let 
$C_1,C_2,C_3 \in  \mathscr{P}'$ with $C_1\equiv C_2$ and $C_1\equiv C_3$. We will prove that $C_3\equiv C_2$.

There exist  $C,C' \in  \mathscr{P}'$ such that   $ f^{\ell}(C_1)\cup  f^k(C_2)\subset C$  and 
 $ f^{p}(C_1)\cup  f^q(C_3)\subset C'$ for some $\ell,k,p,q\geq 0$.
If $\ell \ge p$, then  $f^{\ell-p}(C') \subset C$, which means that $f^{q+\ell-p}(C_3) \subset C$ implying that  $C_3\equiv C_2$. Otherwise $\ell < p$,  then  $f^{p-\ell}(C) \subset C'$, which means that $f^{k+p-\ell}(C_2) \subset C'$ implying that  $C_3\equiv C_2$. We have proved that $\equiv$ is an equivalence relation

Denote by $[C]$ the equivalence class of the interval $C\in \mathscr{P}'$. Now we will prove that $\equiv$ has at most $n$ equivalence classes.

For each $1\le k \le n-1$, let $m_{k}\ge 1$ denote the number of pairwise distinct terms in the sequence
$[F_1'], [G_1'],\ldots, [F_{k}'],[G_{k}']$.
We have that $m_1\le 2$. By Corollary \ref{P3}, for each $2\le k \le n-1$, there exist $C_1\in \{F_1',G_1',\ldots,F_{k-1}',G_{k-1}'\}$ and $C_2\in \{F_{k}',G_{k}'\}$ such that $C_1\equiv C_2$. Hence, $m_{k}\le m_{k-1}+1$ for every $2\le k\le n-1$. By induction, $m_{k}\le k+1$ for every $1\le k \le n-1$.
The proof is finished by taking $k=n-1$.
\end{proof}

\noindent
{\bf Proof of Theorem \ref{IP1}}.
 The fact that $(x_1,\ldots,x_{n-1})\in \Omega'_{n-1}$ implies that the periodic orbits of $f$ are entirely contained in the union of the intervals of the quasi-partition $\mathscr{P}$. Moreover, each interval of $\mathscr{P}$ intersects at most one periodic orbit of $f$. By Lemma \ref{P0}, every orbit of $f$ intersects an interval of $\mathscr{P}'$.  The intervals of $\mathscr{P}'$ that intersect the same periodic orbit of $f$ belong to the same equivalence class. In this way, there exists an injective map that assigns to each periodic orbit of $f$ an equivalence class. By Lemma \ref{P4}, the number of equivalence classes  is at most $n$. As a result, the number of periodic orbits of $f$ is at most $n$. 
\cqd

\section{Proof of Theorem \ref{main}}\label{proofofmain}

We begin with a lemma which will be used to prove Theorem \ref{main} in the case $I=\R$.

\begin{lemma}\label{-rr} Let  $\phi_i:\R\to\R$, $1\le i\le n$, be $\rho$-Lipschitz contractions. Then there exists $r_0=r_0(\phi_1,\ldots,\phi_n)>0$ such that for every $r\ge r_0$, the following holds:
\begin{itemize}
\item [$(i)$] $\phi_i\left( [-r,r]\right)\subset (-r,r)$ for every $1\le i\le n$;
\item [$(ii)$] For every $x\in\R$, there exists $k=k(x)\ge 0$ such that $\vert h(x) \vert<r$, for every $h\in\mathscr{C}_k$.
\end{itemize}
\end{lemma}
\begin{proof} 
Let $c=\max_i \vert \phi_i(0)\vert$ and $1\le i\le n$, then, for every $x\in\R$, the following holds
$$
\vert \phi_i(x)\vert\le\vert \phi_i(x)-\phi_i(0)\vert+\vert\phi_i(0)\vert\le \rho \vert x\vert+ c.
\eqno{(12)}
$$
Set $r_0:=2c/(1-\rho)$ and let $r\ge r_0$. Note that 
$$
\vert x\vert \le r\implies \vert \phi_i(x)\vert\le \rho \vert x\vert+c\le \rho r +\frac{(1-\rho) r}{2}<r
$$
which proves $(i)$.

By $(12)$, let $h=\phi_{i_1}\circ \phi_{i_2}\circ \cdots\circ \phi_{i_k} \in\mathscr{C}_k$, where $1\le i_1,i_2,\ldots,i_k \le n$. We have that
$$ 
\vert h(x)\vert=\vert \phi_{i_1}\circ \phi_{i_2}\circ \cdots\circ \phi_{i_k}( x)\vert\le \rho^k\vert x\vert+(\rho^{k-1}+\ldots+\rho +1)c\le \rho^k\vert x\vert + \dfrac{c}{1-\rho}\le \rho^k \vert x\vert +\dfrac{r}{2}.
$$
Given $x\in\R$, let $k$ be so large that $\rho^k \vert x\vert<r/2$, then $\vert h(x) \vert<r$ which proves $(ii)$.
\end{proof}

\begin{proof}[Proof of Theorem \ref{main}] It follows straightforwardly from Theorems \ref{toconclude2} and  \ref{IP1} that Theorem \ref{main} holds for $I=[0,1)$, thus the same claim holds if instead of $[0,1)$ we consider any bounded real interval of the form $[a,b)$.

Now we consider the case $I=\R$. Let $\phi_1,\ldots,\phi_n:\R\to\R$ be $\rho$-Lipschitz contractions and $r_0=r_0(\phi_1,\ldots,\phi_n)>0$ be given by Lemma \ref{-rr}. 
For every integer $k\ge 0$, set  $I_k=[-(r_{0}+k),r_{0}+k)$.
By the item $(i)$ of Lemma \ref{-rr} , $\phi_i\left(\overline{I}_k\right)\subset \mathring{I_k}$ for every $1\le i\le n$. 
Let  $\phi_i^{(k)}:=\phi_i\vert_{\overline{I}_k}$, then $\{\phi_1^{(k)},\ldots,\phi_n^{(k)}\}$ is an IFS consisting of $\rho$-Lipschitz contractions defined on $\overline{I}_k$. 
Hence, by the first part of the proof, there exists a full subset $V_k$ of $I_k$ such that, for every
$(x_1,\ldots,x_{n-1})\in \Omega_{n-1}(I_k)\cap (V_k)^{n-1}$,  the map $f_{\phi_1^{(k)},\ldots,\phi_n^{(k)},x_1,\ldots,x_{n-1}}:I_k\to\mathring{I_k}$ is asymptotically periodic and has at most $n$ periodic orbits. 
By the items $(i)$ and $(ii)$  of Lemma \ref{-rr}, for every $k\ge 1$,
the maps $f_{\phi_1,\ldots,\phi_n,x_1,\ldots,x_{n-1}}:\R\to\R$ and 
$f_{\phi_1^{(k)},\ldots,\phi_n^{(k)},x_1,\ldots,x_{n-1}}:I_k\to\mathring{I_k}$
have the same asymptotic limits. Therefore $f_{\phi_1,\ldots,\phi_n,x_1,\ldots,x_{n-1}}$
 is also asymptotically periodic and has at most $n$ periodic orbits. 
 
 To conclude the proof, let $W_k= (-\infty, -(r_{0}+k))\cup V_k \cup ( r_{0}+k,\infty)$. Therefore $W_k$ is a full subset of $\R$ and the denumerable intersection
 $$
 W_{\Phi}= \cap_{k\ge 1} W_k
 $$
 is also a full subset of $\R$. Let $(x_1,\ldots,x_{n-1})\in \Omega_{n-1}(\R)\cap (W_{\Phi})^{n-1}$ and
  $k$ be an integer larger than $\max\{\vert x_1 \vert, \vert x_{n-1} \vert \}$. Thus, the point $(x_1,\ldots,x_{n-1})$ also belongs to the set $ \Omega_{n-1}(I_k)\cap (V_k)^{n-1}$, implying that 
 the map  $f_{\phi_1,\ldots,\phi_n,x_1,\ldots,x_{n-1}}$
 is  asymptotically periodic and has at most $n$ periodic orbits.
 This concludes the proof of the theorem.

 \end{proof}
 
 \section{Proof of Theorem \ref{abcd}}\label{proofofabcd}
 
 Throughout this section, let $-1<\lambda<1$ and
 $\phi_1,\ldots,\phi_n:\R\to\R$ be $\lambda$-affine maps  defined by
 $\phi_i(x)=\lambda x + b_i$, where  $b_1,\ldots,b_n\in\R$. Hereafter, to avoid misunderstanding,  
 whenever a piecewise $\lambda$-affine map is defined on the whole line, we use the notation ${\bar{f}}_{\phi_1,\ldots,\phi_n,x_1,\ldots,x_{n-1}}$ in  place of ${f}_{\phi_1,\ldots,\phi_n,x_1,\ldots,x_{n-1}}$.
 
\begin{lemma}[Reduction Lemma]\label{redl} Let $(c_1,\ldots,c_{n-1})\in\Omega_{n-1}(\R)$ and $\bar{f}={\bar f}_{\phi_1,\ldots,\phi_n,c_1,\ldots,c_{n-1}}$. Then, for every $\delta\in\R$, the map ${\bar{f}}_{\delta}:\R\to\R$ defined by 
${\bar f}_{\delta}={\bar f}+\delta$
   is topologically conjugate to the map ${\bar g}={\bar f}_{\phi_1,\ldots,\phi_n,x_1,\ldots,x_{n-1}}$, where 
   $x_i=c_i-\delta/(1-\lambda)$ for every $1\le i\le n-1$. 
   \end{lemma}
 \begin{proof} 
 Let $\delta\in\R$ be fixed. Set $x_i=c_i-\delta/(1-\lambda)$ and
  ${\bar g}={\bar f}_{\phi_1,\ldots,\phi_n,x_1,\ldots,x_{n-1}}$. 
 Let $h:\R\to\R$ be defined by $h(x)=x+\delta/(1-\lambda)$. 
 We claim that $h\circ {\bar g}={\bar f}_{\delta}\circ h$. To show this, let $I_1=(-\infty,x_1), I_n=[x_{n-1},\infty)$ and
 $I_j=[x_{j-1},x_j)$, $2\le j\le n-1$. 
 By the definition of $h$ and $x_i$, we have that $h(I_1)=(-\infty,c_1)$, $h(I_n)=[c_{n-1},\infty)$ and $h(I_j)=[c_{j-1},c_j)$ for every $2\le j \le n-1$. Moreover, by $(\ref{mapf})$,
  for every $x\in I_i$, $1\le i \le n$, we have that 
 $$h\left({\bar g}(x)\right)=h\left(\phi_i(x)\right)=\phi_i(x)+\dfrac{\lambda\delta}{1-\lambda}+\delta=\phi_i\left(x+\dfrac{\delta}{1-\lambda} \right)+\delta=\phi_i\left(h(x)\right)+\delta={\bar f}_{\delta}\left(h(x)\right).$$
This proves the claim.
\end{proof}
      
\begin{proof}[Proof of Theorem \ref{abcd}] Let $I=[0,1)$ and $f:I\to \R$ be an $n$-interval piecewise $\lambda$-affine contraction, then there exist $\lambda$-affine contractions $\phi_1,\ldots,\phi_n:\R\to\R$ and points
$0=d_0<d_1<\ldots <d_{n-1}<d_n=1$ such that
      $f(x)=\phi_i(x)$ for every $x\in [d_{i-1},d_i)$. Let $F$ be the denumerable set 
      $$  F=\bigcup_{i=1}^n\bigcup_{j=0}^1\left\{\delta\in\R: \phi_i(d_{i-j})+\delta=0\,({\rm mod}\,1)\right\},$$
      then for each $\delta_0\in \R\setminus F$, there exist $\epsilon>0$, $n\le m\le 2n$ and $(c_1,\ldots,c_{m-1})\in \Omega_{m-1}(I)$ such that 
      \begin{equation}\label{v4}
      f_{\delta}=f+\delta\,({\rm mod}\, 1)=f_{\phi_1,\ldots,\phi_{m},c_1,\ldots,c_{m-1}}+\delta\quad\textrm{for every}\quad \delta\in (\delta_0-\epsilon,\delta_0+\epsilon).
  \end{equation}    
       Let $\bar{f}=\bar{f}_{\phi_1,\ldots,\phi_m,c_1,\ldots,c_{m-1}}:\R\to\R$. 
By Lemma \ref{redl}, for every $\delta\in\R$, the map ${\bar f}_\delta=\overline{f}+\delta$ is topologically conjugate to the map ${\bar f}_{\phi_1,\ldots,\phi_m,x_1,\ldots,x_{m-1}}:\R\to\R$, where $x_i=c_i-\delta/(1-\lambda)$ for every $1\le i\le m-1$. Let $W_{\Phi}$ be the full set given by Theorem \ref{main}. Let $\delta$ belong to the full set $\bigcap_{i=1}^n  (1-\lambda)\left(c_i-W_{\Phi}\right)$, then $(x_1,\ldots,x_{m-1})=\left(c_1-\dfrac{\delta}{1-\lambda},\ldots,c_{m-1}-\dfrac{\delta}{1-\lambda} \right)\in \Omega_{m-1}(\R)\cap W_{\Phi}^{m-1}$, therefore, by Theroem \ref{main}, the map
   ${\bar f}_{\phi_1,\ldots,\phi_m,x_1,\ldots,x_{m-1}}$ is asymptotically periodic and has at most $m\le 2n$ periodic orbits. 
  The map $\bar{f}_{\delta}$ inherits from ${\bar f}_{\phi_1,\ldots,\phi_m,x_1,\ldots,x_{m-1}}$ the same asymptotic properties.  
By $(\ref{v4})$, $f_{\delta}(x)={\bar f}_\delta(x)$ for every $x\in I$ and  $\delta\in (\delta_0-\epsilon,\delta_0+\epsilon)$.   
In this way, $f_{\delta}$ is asymptotically periodic and has  at most $m\le 2n$ periodic orbits for almost every $\delta\in (\delta_0-\epsilon,\delta_0+\epsilon)$.  
\end{proof}

\end{document}